\documentclass[a4paper]{amsart}

\usepackage{amsmath,amstext,amssymb,mathrsfs,amscd,amsthm,indentfirst}
\usepackage{amsfonts}
\usepackage{enumerate}
\usepackage{verbatim}
\usepackage{pstricks}
\usepackage{graphicx}
\usepackage{stmaryrd}
\usepackage{bbm}
\usepackage{lipsum}
\usepackage{hyperref}

\usepackage{tikz, tikz-cd}

\theoremstyle{plain}
\newtheorem{theorem}{Theorem}

\newtheorem{lemma}[theorem]{Lemma}

\theoremstyle{definition}

\numberwithin{equation}{section}

\newcommand{\cat}[1]{\mathsf{#1}}

\newcommand{\bR}{\mathbb{R}}
\newcommand{\bZ}{\mathbb{Z}}

\newcommand{\Sp}{\mathrm{Sp}}

\newcommand{\Diff}{\mathrm{Diff}}

\newcommand{\map}{\operatorname{map}}

\newcommand{\Aut}{\mathrm{Aut}}

\newcommand\lra{\longrightarrow}

\newcommand\End{\mathrm{End}}

\newcommand\Ker{\operatorname*{Ker}}

\newcommand{\Top}{\mathrm{Top}}
\newcommand{\G}{\mathrm{G}}
\newcommand{\OO}{\mathrm{O}}
\newcommand{\fr}{\mathrm{fr}}

\title{TQFTs do not detect the Milnor sphere}

\author{Ben Gripaios}
\email{bmg33@cam.ac.uk}
\address{Cavendish Laboratory \\ University of Cambridge \\
JJ Thomson Avenue \\ 
Cambridge CB3 0US\\ UK}

\author{Oscar Randal-Williams}
\email{or257@cam.ac.uk}
\address{Centre for Mathematical Sciences\\
Wilberforce Road\\
Cambridge CB3 0WB\\
UK}

\begin{document}

\begin{abstract}
  We show that, under very general hypotheses, topological quantum
  field theories (TQFTs) cannot detect homotopy spheres bounding
  parallelisable manifolds, such as Milnor's exotic
  7-dimensional sphere. The result holds for a wide variety of target
  categories (or $(\infty,n)$-categories) and arbitrary tangential structures. An appendix contains
  results on the mapping class groups of (stably-) framed manifolds
  that may be of independent interest.  
  \end{abstract}
	
	\maketitle
 
An interesting question for mathematicians and physicists alike is: to
what extent do functorial topological quantum field theories (TQFTs) detect
exotic smooth manifolds? Indeed, mathematicians might hope to be able
to use the cutting and pasting properties of a suitable TQFT to decide
whether or not an exotic 4-sphere exists, while physicists would like to understand more generally the relevance of exotic structures for physics. 

There has been some progress on this question recently, which may be summarised as follows. 
On the one hand, in \cite{Reutterdont} (in which  a TQFT was defined to be a
symmetric monoidal functor from the oriented bordism category to the
category of vector spaces over an algebraically-closed
field), it was shown that no
4-dimensional semisimple TQFT 
 can detect exotic 4-manifolds. On the other hand, 
it was shown in \cite{Reutterdo} that in higher dimensions there exist
semisimple TQFTs (now valued in supervector spaces and either with or
without an orientation) that detect Hitchin spheres, i.e.\ exotic spheres
that do not bound spin manifolds, and which exist precisely in
dimensions $8k+1$ and $8k+2$, for $k \geq 1$.

These results prompted a question in \cite[Question 1.3]{Reutterdo}, whose most general version is: do semisimple TQFTs detect all exotic spheres in dimension $>4$? 

Here we show that the answer is: no. A counterexample is Milnor's
prototypical exotic 7-sphere (which bounds a parallelisable, ergo
  spinnable, manifold). 
In fact, our results are much stronger, in that they: (i) require no
semisimplicity hypothesis; (ii) apply not
just to spheres, but rather to arbitrary bordisms; (iii) hold for a large class of target
categories generalising that of vector spaces;  (iv) hold for 
bordism
categories in which the manifolds are equipped not just with an
orientation, but with an
arbitrary tangential structure; (v) also apply to TQFTs that
are extended upwards and/or downwards. The last of these allows us to
show that so-called `cohomological TQFTs', taking values in the
$(\infty,1)$-category of chain complexes of vector spaces, do not detect the Milnor
7-sphere either.

We begin with a simple statement relevant for physics, as
follows.
\begin{theorem} \label{thm:nofrills}
Let $\Sigma$ be a $(4k-1)$-dimensional oriented
homotopy sphere that
bounds a parallelisable
$4k$-dimensional manifold, and let $F : \cat{Bord}^{SO}_{4k-1} \to \cat{Vect}_\Bbbk$ be an oriented TQFT
taking values in the category of vector spaces over a field
$\Bbbk$. Then for any nonempty $(4k-1)$-dimensional oriented bordism $ M : P \leadsto Q$ we have
$$F(M \# \Sigma) = F(M) : F(P) \lra F(Q)$$
for any choice of connected-sum.
\end{theorem}
In particular, taking $4k-1=7$, $M = S^7 : \emptyset \leadsto \emptyset$ and $\Sigma$ to be the
Milnor 7-sphere, we see that such a TQFT satisfies $F(\Sigma) =F(S^7) \in \Bbbk$, giving the statement in the title of the paper.

Before describing the various generalisations explicitly, we
give a proof, based on an idea recently
exploited by Krannich, Kupers, and Mezher \cite{KKM}.

\begin{proof}
The statement is trivial for $4k-1=3$ as by the Poincar{\'e} conjecture the only homotopy sphere in this dimension is $S^3$, and connect-sum with this has no effect. So we may suppose that $4k-1 \geq 7$, so that we are squarely in the realm of high-dimensional manifold theory.

Choose any disc $D^{4k-1} \subset \mathrm{int}(M)$. The handlebody $V_g := \natural_g S^{2k-1} \times D^{2k}$ embeds into
$D^{4k-1}$ and thereby embeds into the interior of the bordism $M$. Given such an embedding, let
$M'$ be the complement of the interior of $V_g$. Then the bordism
$M : P \leadsto Q$ can be factored as the composition
$$M : P \overset{M'} \leadsto Q \bigsqcup \partial V_g \overset{\mathrm{Id} \sqcup V_g}\leadsto Q.$$
If $\varphi$ is an
oriented diffeomorphism of $\partial V_g$, then we may instead
glue $V_g$ to $M'$ along $\varphi$, giving a bordism 
$$M' \cup_{\varphi} V_g : P \overset{M'} \leadsto Q \bigsqcup \partial
V_g \overset{\mathrm{Id} \sqcup \varphi}\leadsto Q \bigsqcup \partial
V_g  \overset{\mathrm{Id} \sqcup V_g}\leadsto Q.$$
The strategy will be to show that for $g \geq 2$ the diffeomorphism $\varphi$ of $\partial V_g$ can be chosen so that
\begin{enumerate}[(a)]
 \item there is an oriented diffeomorphism $M' \cup_{\varphi} V_g \cong M \# \Sigma$ relative to $P$ and $Q$, and
 
 \item $F(\varphi) = \mathrm{Id} : F(\partial V_g) \to F(\partial V_g)$.
 
 \end{enumerate}
Given these, functoriality and monoidality of $F$ imply that $F(M \# \Sigma) = F(M)$.

The group $\Theta_{d}$ of $d$-dimensional oriented homotopy spheres
may
be identified with $\pi_0\mathrm{Diff}_\partial(D^{d-1})$, via the
following construction. A diffeomorphism $\psi$ of $D^{d-1}$ which is the identity near the boundary extends
to a oriented diffeomorphism $\psi'$ of $S^{d-1}$, which may be used to glue
two copies of the $d$-disc together to form an oriented exotic sphere $\Sigma_{\psi} :=
D^{d} \cup_{\psi'} D^{d}$. This defines a function
$\pi_0\mathrm{Diff}_\partial(D^{d-1}) \to \Theta_{d}$ which using Smale's $h$-cobordism theorem and Cerf's pseudoisotopy-implies-isotopy theorem can be checked to be an isomorphism as long as $d\geq 6$. 

Similarly, the
diffeomorphism $\psi$ extends to an oriented diffeomorphism
$\psi''$ of the manifold $W_g := \partial V_g = \#_g S^{2k-1} \times
S^{2k-1}$, which can be used to form the oriented cobordism $M' \cup_{\psi''} V_g$. We now explain how to identify this manifold as the connected-sum $M \# \Sigma_\psi$. Attaching $2k$-handles along the $g$ disjoint copies of $D^{2k-1} \times S^{2k-1}$ inside $\partial V_g \setminus \mathrm{int}(D^{4k-2})$ gives a cobordism $U : \partial V_g \leadsto S^{4k-2}$, and $V_g = U \cup_\mathrm{Id} D^{4k-1}$. The diffeomorphism $\psi''$ is the identity on $\partial V_g \setminus \mathrm{int}(D^{4k-2})$, so it extends to a diffeomorphism of $U$ which is $\psi''$ on the incoming boundary $\partial V_g$ and $\psi'$ on the outgoing boundary $S^{4k-2}$. This induces an oriented diffeomorphism
$$M' \cup_{\psi''} V_g = M' \cup_{\psi''} (U \cup_\mathrm{Id} D^{4k-1}) \cong (M' \cup_\mathrm{Id} U ) \cup_{\psi'} D^{4k-1} = M \# \Sigma_{\psi},$$
giving (a).

For (b) we use the following from \cite[Theorem 2.3]{KKM}, which systematises \cite{KRW}: If $[\Sigma_\psi] \in \Theta_{4k-1}$ is a homotopy sphere which bounds a parallelisable $4k$-dimensional manifold, and $g \geq 2$, then the corresponding diffeomorphism
$$[\psi''] \in \pi_0 \mathrm{Diff}^+(W_g)$$
is in the finite residual of this group,  i.e.\ is contained in every finite index subgroup. We are free to take $g$ as large as we like, so may assume that this applies.

We will also use that the group $\pi_0 \mathrm{Diff}^+(W_g)$ is
finitely-generated, which follows from \cite[Theorem 2]{kreckisotopy}.

To establish (b) we consider the homomorphism
$$F : \pi_0 \mathrm{Diff}^+(W_g) \lra \mathrm{GL}(F(W_g))$$
induced by $F$. Its image is finitely-generated (as $\pi_0
\mathrm{Diff}^+(W_g)$ is) and is a linear group (tautologically, as the vector space
$F(W_g)$ is dualisable and hence finite-dimensional), so by a theorem
of Mal'cev \cite{Malcev} its image is residually finite. As $[\psi'']$ is in
the finite residual of $\pi_0 \mathrm{Diff}^+(W_g)$, it is therefore
in the kernel of this homomorphism, i.e.\ $F(\psi'') = \mathrm{Id}$, giving (b).
\end{proof}

Now we consider the ways in which the Theorem can be
generalised,
considering the directions (iii-v)  in turn.
\section*{Other target categories}
Clearly it is not particularly important in the proof of Theorem \ref{thm:nofrills}
that the target of the TQFT is the category of vector spaces over a
field $\Bbbk$. As we
will now show, the 
Theorem remains true for a much larger class of
target symmetric monoidal categories, including obvious generalisations from algebra (e.g.\
modules over an arbitrary ring or graded versions thereof),
homological algebra  
(e.g.\ chain complexes of vector spaces), and algebraic
geometry (e.g.\ quasicoherent sheaves on schemes). 
Examples of obvious relevance for physics are the category of
super vector spaces (which allow us to encode fermions in physics) or
the derived category of chain complexes of vector spaces (which are
relevant for cohomological TQFTs).

To give a general statement, we need to introduce a few
elementary notions regarding groups and categories. To wit, recall that a group is \emph{residually finite} if 
any two distinct elements have distinct images under some homomorphism
to a finite group and that a group is \emph{locally residually finite}
if every finitely-generated subgroup is residually finite. We say an
object in a category is \emph{round} if its group of automorphisms is
locally residually finite and define a \emph{well-rounded} category to
be a symmetric monoidal category in which every dualisable object is
round.

The conclusion of Theorem \ref{thm:nofrills} continues to hold if we replace the target of the TQFT by any
well-rounded category. 

The restriction on a symmetric monoidal category of being well-rounded is a weak one because the class of locally residually
finite groups is large. For example, any finite group is residually finite (via the
identity homomorphism), as is the group of integers (because the
homomorphism $\bZ \to \bZ/m$ distinguishes every integer with
absolute value less than $m$ from zero ). The property of being
residually finite is easily seen to be preserved under taking
subgroups and products, as is the property of being locally residually finite.
It follows that any limit of residually finite (resp. locally
residually finite) groups in the category of groups is residually finite (resp. locally
residually finite), since an equaliser of a pair of maps is a
subgroup of the source.
So in fact, every finitely-generated
abelian group is residually finite and every 
abelian group is locally residually finite. It turns out that
profinite groups, free groups, solvable groups, and nilpotent groups
are all residually finite and we already saw in the proof of Theorem
\ref{thm:nofrills} that every linear group 
(i.e.\ a group isomorphic to a subgroup of $GL_n(\Bbbk)$
for some $n$ and some field $\Bbbk$) is locally residually finite.

So
the category of vector spaces over any field $\Bbbk$ is well-rounded,
because the dualisable objects are precisely the finite-dimensional
vector spaces. More generally, we have the following.
\begin{theorem}\label{thm:round}
  The following are well-rounded categories.
\begin{enumerate}
    \item The category of modules over a ring (or equivalently, the
      category of quasicoherent sheaves on an affine scheme).
    \item The category of graded modules over a graded ring.
  \item The derived category of vector spaces over a field.
    \item The category of quasicoherent sheaves on a scheme.
      \end{enumerate}
  \end{theorem}

  \begin{proof}
    We proceed item-by-item.
    \begin{enumerate}
      \item The dualisable objects are the finitely-generated projective
        modules \cite[Example 1.4]{DoldPuppe} and any finitely-generated module is round
        \cite[6.2]{wehrfritz}.
        \item A dualisable object is a
summand of a finite sum of shifts of
$R$ and its
automorphisms are a subgroup of its automorphisms considered as a 
finitely-generated module over the ring obtained by
forgetting the grading. Hence they form a locally residually finite
group by (1).
  \item The derived category of vector spaces is
    equivalent to the category of graded vector spaces
    \cite{Keller}, so is well-rounded by (2). 
     \item For any scheme $X$, the dualisable objects of $\cat{QCoh}(X)$
       are the locally-free sheaves of finite type, or equivalently
       vector bundles on $X$. Given such a vector bundle $E$, let
       $\End(E)$ denote the corresponding bundle of endomorphisms and let
       $\{X_\alpha\}$ be a cover of $X$ by affine schemes.  The restriction map
    $H^0(X ; \End(E)) \to \prod_\alpha H^0\big(X_\alpha ; \End(E\vert_{X_\alpha})\big)$
is injective (by the sheaf property) and this gives an injective map on units
    $H^0(X ; \End(E))^\times \to \prod_\alpha H^0\big(X_\alpha ; \End(E\vert_{X_\alpha})\big)^\times.$
But if $X_\alpha = \mathrm{Spec}(A_\alpha)$ then the vector bundle $E\vert_{X_\alpha}$ corresponds to a finitely-generated projective $A_\alpha$-module $M_\alpha$ and so 
    $H^0\big(X_\alpha ; \End(E\vert_{X_\alpha})\big)^\times = \Aut_{A_\alpha}(M_\alpha),$
which is locally residually finite by (1). Thus $H^0(X ;
\End(E))^\times$ embeds into a product of locally residually finite
groups so is itself locally residually finite.\qedhere
    \end{enumerate}
  \end{proof}

Clearly we can detect the Milnor
sphere (and all exotic manifolds) in principle if we are given arbitrary freedom in
the choice of target category, since we can take our TQFT to be the identity functor. The
question then arises of what kind of target
category might allow us to re-prove the existence of the Milnor sphere.
\section*{Other tangential structures}
Theorem \ref{thm:nofrills} continues to hold when
the bordism category has objects and
  morphisms equipped with arbitrary tangential structure. Let $\theta : B \to B\OO(4k-1)$ be a tangential structure, classifying a vector bundle $\theta^*\gamma_{4k-1} \to B$, and let
$\cat{Bord}^{\theta}_{4k-1}$ be its associated bordism category.
It suffices to see that if $(M,\ell_M) : (P, \ell_P) \leadsto
(Q,\ell_Q)$ is a $\theta$-bordism, $V_g \subseteq D^{4k-1} \subseteq
M$ is an embedding, and $\ell$ is the $\theta$-structure on $W_g =
\partial V_g = \#^g S^{2k-1} \times S^{2k-1}$ obtained by restricting
$\ell_M$, then the $\theta$-structure mapping class
group $$\Gamma^\theta(W_g,\ell) := \pi_1\big(\mathrm{Bun}(TW_g \oplus \bR,
\theta^*\gamma_{4k-1}) /\!\!/ \mathrm{Diff}(W_g) , \ell\big)$$
has the following properties:
\begin{enumerate}[(a)]

\item there is an element $[\psi''_\theta] \in
  \Gamma^\theta(W_g,\ell)$ which on neglecting $\theta$-structures maps to $[\psi''] \in
  \pi_0\mathrm{Diff}(W_g)$, and

\item $[\psi''_\theta]$ is in the finite residual of
  $\Gamma^\theta(W_g,\ell)$ for large enough $g$.

\end{enumerate}
By the same argument as the proof of Theorem \ref{thm:nofrills}, we would
then have the following general theorem.
\begin{theorem} \label{thm:frills}
Let $\Sigma$ be a $(4k-1)$-dimensional homotopy sphere that
bounds a parallelisable $4k$-dimensional manifold and let $F :
\cat{Bord}^{\theta}_{4k-1} \to \cat{C}$ be a TQFT with
$\theta$-structure taking values in a well-rounded category $\cat{C}$.
Then for
any nonempty $(4k- 1)$-dimensional $\theta$-bordism $(M, \ell_M) : (P, \ell_P) \leadsto
(Q, \ell_Q)$,
$$F(M \# \Sigma, \ell) = F(M, \ell_M) : F(P, \ell_P) \lra F(Q, \ell_Q)$$
for any choice of connected-sum, where the ${\theta}$-structure $\ell$ on $M \# \Sigma$ depends on the choice of lift $[\psi''_\theta]$ in (a) above but agrees with $\ell_M\vert_{M \setminus \mathrm{int}(D^{4k-1})}$ on $M \setminus \mathrm{int}(D^{4k-1})$.
\end{theorem}

Note that we do not require in Theorem~\ref{thm:frills} that the homotopy sphere
$\Sigma$ itself be equipped with, nor even admit, a $\theta$-structure: it is a consequence of the proof that $M \# \Sigma$ admits a $\theta$-structure when $M$ does. But this is not surprising, as it follows from the well-known \cite{EbertMO} fact that the tangent bundles of $\Sigma$ and $S^{4k-1}$ correspond under the homotopy equivalence (or homeomorphism) between them, so the tangent bundles of $M \# \Sigma$ and $M$ also correspond under the homotopy equivalence (or homeomorphism) between them.

Of course $M \# \Sigma$ may also admit other $\theta$-structures $\ell'$ such that $F(M \# \Sigma,\ell') \neq F(M, \ell_M)$, but this is hardly related to whether $F$ can detect the difference between the underlying manifolds $M \# \Sigma$ and $M$. One way to see this is to observe that $M$ can admit another $\theta$-structure $\ell'_M$ such that $F(M,\ell'_M) \neq F(M, \ell_M)$ but of course the underlying manifolds are equal\footnote{A simple concrete example is the $(3+1)$-dimensional invertible TQFT corresponding to the signature: changing orientations negates the signature.}. Indeed a TQFT defined on $\theta$-manifolds detects something about the underlying manifold but also something about the $\theta$-structure, and the formulation in Theorem \ref{thm:frills} should be read as saying that a TQFT cannot tell the difference between the underlying manifolds $M \# \Sigma$ and $M$.

We now explain how to reduce the verification of properties (a) and
(b) to a better-studied situation, namely the framed mapping class
group of the manifold with boundary $W_{g,1} := W_g \setminus
\mathrm{int}(D^{4k-2})$. In the Appendix we will verify these
properties in this case, which may be of independent interest, and this concludes the proof of Theorem \ref{thm:frills}.

\begin{lemma}\label{lem:destab}
There is a framing ${\ell}^\mathrm{fr} : TW_{g,1} \to \bR^{4k-2}$
of the manifold with boundary $W_{g,1}$ such that the bundle map
$$ TW_{g,1} \oplus \bR \overset{{\ell}^\mathrm{fr} \oplus \bR}\lra \bR^{4k-1} \overset{\text{fibre inclusion}}\lra \theta^*\gamma_{4k-1}$$
is homotopic to the restricted $\theta$-structure $\ell\vert_{W_{g,1}}$.
\end{lemma}

Via this lemma there is a change-of-tangential-structure-then-glue-in-the-disc homomorphism
$$\Gamma_\partial^\mathrm{fr}(W_{g,1}, \ell^\mathrm{fr}) \lra \Gamma^\theta(W_g, \ell),$$
using which properties (a) and (b) for $\Gamma_\partial^\mathrm{fr}(W_{g,1}, \ell^\mathrm{fr})$ imply the same for $\Gamma^\theta(W_g, \ell)$.

\begin{proof}[Proof of Lemma \ref{lem:destab}]
The restriction $\ell$ of $\ell_M$ from $M$ to $W_g$ is the same as the restriction of $\ell_M\vert_{D^{4k-1}}$ to $W_g$, but as $D^{4k-1}$ is contractible the $\theta$-structure $\ell_M\vert_{D^{4k-1}}$ arises from a framing of $D^{4k-1}$, and hence the $\theta$-structure $\ell$ on $W_g$ arises from a (once-stable) framing of this manifold, i.e.\ a bundle map $\ell^{1\text{-}\mathrm{fr}} : TW_g \oplus \bR \to \bR^{4k-1}$. The restriction $\ell^{1\text{-}\mathrm{fr}}\vert_{W_{g,1}} : TW_{g,1}\oplus \bR \to \bR^{4k-1}$ can be homotoped to a bundle map of the form $\ell^\mathrm{fr} \oplus \bR$, as when formulated as a lifting problem this corresponds to lifting a nullhomotopy of the composition
$$W_{g,1} \overset{\tau_{W_{g,1}}}\lra B\OO(4k-2) \lra B\OO(4k-1)$$
to a nullhomotopy of $\tau_{W_{g,1}}$, and this is possible as
$W_{g,1}$ is homotopy equivalent to a $(2k-1)$-dimensional cell
complex and the map $B\OO(4k-2) \to B\OO(4k-1)$ is
$(4k-3)$-connected.
\end{proof}
\section*{Extended TQFTs}
The locality of the known laws of physics suggests that TQFTs relevant
for physics should be fully extended downwards, in the sense of being defined on
manifolds of all dimensions up to the dimension of spacetime being
considered. This requires the use of $n$-categories. Moreover, the
fact that many of the explicit TQFTs encountered in physics are cohomological
in nature suggests that TQFTs should also be extended upwards, in the
sense of being defined not as functors between categories (or
$n$-categories), but between $(\infty,1)$-categories (or $(\infty,n)$-categories).

Our results can be applied in such cases as well. To see how, observe
that given any TQFT extended downwards, we can obtain an unextended
TQFT by restricting all manifolds of codimension greater than one to
be the empty set, to which the Theorems already derived can be
applied. Similarly, given a TQFT that is extended upwards, we can pass
to the corresponding homotopy categories, whose morphisms are
equivalence classes of 1-morphisms under the equivalence relation of
being connected by a 2-morphism. Assuming the homotopy category of the
target is well-rounded, our Theorems will apply.

In particular, consider cohomological TQFTs taking values in the derived $(\infty,1)$-category of a field $\Bbbk$, whose corresponding homotopy category is the classical derived category of vector spaces over $\Bbbk$. The latter is well-rounded by part (3) of Theorem~\ref{thm:round}, so the theorem in the title of this paper applies to 
such cohomological TQFTs as well.

\appendix
\section*{Appendix: Finite residuals of (stably-) framed mapping class groups \label{app:mcg}}
Our goal in this appendix is to outline the proof of the following statement, using the results and methods of \cite{GRWAb, K, KRWframings}. We will necessarily assume familiarity with these.

Let $\ell^\mathrm{fr}$ be a framing of $W_{g,1}$ and $\Gamma_\partial^\mathrm{fr}(W_{g,1} , \ell^\mathrm{fr})$ be the framed mapping class group.

\begin{theorem}
Let $[\psi] \in \pi_0\Diff(D^{4k-2})$  correspond to a homotopy sphere $[\Sigma_\psi]$ lying in the subgroup $\mathrm{bP}_{4k} \leq \Theta_{4k-1}$, and $[\psi''] \in \pi_0\Diff_\partial(W_{g,1})$ denote the implantation of this diffeomorphism into $W_{g,1}$. Then 

\begin{enumerate}[(a)]

\item there is an element $[\psi''_\theta] \in \Gamma_\partial^\mathrm{fr}(W_{g,1} , \ell^\mathrm{fr})$ which maps to $[\psi''] \in \pi_0\mathrm{Diff}_\partial(W_{g,1})$, and

\item $[\psi''_\theta]$ is in the finite residual of $\Gamma_\partial^\mathrm{fr}(W_{g,1} , \ell^\mathrm{fr})$ for large enough $g$.

\end{enumerate}
\end{theorem}

We will first prove the analogous statement for the stably-framed mapping class group $\Gamma_\partial^\mathrm{sfr}(W_{g,1} , \ell^\mathrm{sfr})$, and then  deduce the statement for $\Gamma_\partial^\mathrm{fr}(W_{g,1} , \ell^\mathrm{fr})$ from this.

\subsection*{The stably-framed mapping class group}

Let $\Theta_{d}^\mathrm{sfr}$ denote the group of homotopy $d$-spheres
equipped with a choice of stable framing. Just as smoothing theory
identifies $\Theta_d = \pi_d(\Top/\OO)$ (for $d \neq 4$), it
identifies $\Theta_{d}^\mathrm{sfr} = \pi_{d}(\Top)$. The analogue of
the Kervaire--Milnor sequence in this setting is the long exact
sequence associated to the fibration sequence $\Top \to \G \to
\G/\Top$, which using Sullivan's calculation of $\pi_d(\G/\Top)$ (in terms of quadratic $L$-theory) gives in particular an exact sequence
$$0 \lra \mathrm{bP}_{4k}^\mathrm{sfr} := \pi_{4k}(\G/\Top) = \bZ \lra \Theta_{4k-1}^\mathrm{sfr} \lra \pi_{4k-1}^s \lra 0.$$
Neglecting stable framings gives a map to the Kervaire--Milnor sequence
$$0 \lra \mathrm{bP}_{4k}\lra \Theta_{4k-1} \lra \pi_{4k-1}^s/(\mathrm{Im}J) \lra \cdots,$$
and $\bZ= \mathrm{bP}_{4k}^\mathrm{sfr} \to \mathrm{bP}_{4k}$ is tautologically surjective.

Just as the mapping class group of $D^d$ may be identified with
$\Theta_{d+1}$, the group
$\Gamma^{\mathrm{sfr}}_\partial(D^{4k-2},\ell_D)$ may be identified
with $\Theta_{4k-1}^\mathrm{sfr}$. Extending stably-framed diffeomorphisms along $D^{4k-1} \subset W_{g,1}$ by the identity defines a homomorphism
$$\Theta_{4k-1}^\mathrm{sfr} = \Gamma^{\mathrm{sfr}}_\partial(D^{4k-2},\ell_D) \lra \Gamma^{\mathrm{sfr}}_\partial(W_{g,1},\ell_W),$$
which lands in the centre. (As any stably-framed diffeomorphism of $W_{g,1}$ which fixes the boundary can be isotoped to also fix $D^{4k-1} \subset W_{g,1}$.) It is also injective. (As re-gluing $V_g$ to $S^{4k-1} \setminus \mathrm{int}(V_g)$ along an element in the kernel gives $S^{4k-1}$ with its ordinary framing, but it also gives the element of $\Theta_{4k-1}^\mathrm{sfr}$ that one started with.)



\begin{lemma}\label{lem:Spq}
Acting on $H_{2k-1}(W_{g,1};\bZ)$ induces an isomorphism
$$\Gamma^{\mathrm{sfr}}_\partial(W_{g,1},\ell_W)/\Theta_{4k-1}^\mathrm{sfr} \overset{\sim}\lra
\begin{cases}
\Sp_{2g}^q(\bZ) & 2k-1 \neq 3, 7\\
\Sp_{2g}^{q \text{ or } a}(\bZ) & 2k-1 = 3, 7.
\end{cases}
$$
to a subgroup of $\Sp_{2g}(\bZ)$ consisting of those matrices which preserve a certain quadratic refinement (which has Arf invariant 0 if $2k-1 \neq 3, 7$ and may have either Arf invariant in the remaining cases).
\end{lemma}
\begin{proof}
We will use \cite{KRWframings}, and will assume familiarity with the notation of that paper: for compatibility we temporarily denote $\check{\Gamma}_g^{\mathrm{sfr},\ell_W} = \Gamma^{\mathrm{sfr}}_\partial(W_{g,1},\ell_W)$. We develop the following commutative diagram.

\begin{equation*}
\begin{tikzcd}[column sep = 3ex]
 & 0 \dar & \pi_1\mathrm{Diff}_\partial(W_{g,1}) \dar \arrow[rd, dashed] \\
0 \rar & \pi_{4k-1}(\mathrm{SO}) \rar \dar & \pi_1 \mathrm{map}_\partial(W_{g,1}, \mathrm{SO}) \dar \rar & \pi_1 \mathrm{map}_*(W_{g,1}, \mathrm{SO}) \rar & 0 \\
0 \rar & \Theta_{4k-1}^\mathrm{sfr} \rar \arrow[d, two heads] & \check{\Gamma}_g^{\mathrm{sfr},\ell_W} \arrow[d, two heads] \rar & \check{\Gamma}_g^{\mathrm{sfr},\ell_W}/\Theta_{4k-1}^\mathrm{sfr} \rar \dar & 0\\
0 \rar & \Theta_{4k-1} \rar & \Gamma^{\mathrm{sfr}, [\ell_W]}_g \rar & \Gamma^{\mathrm{sfr}, [\ell_W]}_g/\Theta_{4k-1} \rar & 0
\end{tikzcd}
\end{equation*}

The dashed map is surjective, as follows. Associated to $\xi \in \pi_{2k}(\mathrm{SO}(2k)) = \pi_1(\Omega^{2k-1} \mathrm{SO}(2k))$ there is a loop of diffeomorphisms of $D^{2k-1} \times S^{2k-1}$ which are the identity on the boundary, and this may be implanted into $W_{g,1}$ along any generator $x$ of $\pi_{2k-1}(W_{g,1}) = \bZ\{a_1, b_1, a_2, b_2, \ldots, a_g, b_g\}$ to give elements of $\pi_1\mathrm{Diff}_\partial(W_{g,1})$. Under the dashed map to $\pi_1 \mathrm{map}_*(W_{g,1}, \mathrm{SO}) \cong H_{2k-1}(W_{g,1} ; \pi_{2k}(\mathrm{SO}))$ this is sent to $x$ times the image of $\xi$ under stabilisation $\pi_{2k}(\mathrm{SO}(2k)) \to \pi_{2k}(\mathrm{SO})$. This stabilisation map is surjective for all $k$ (the target is mostly zero, and when it is not use \cite[Theorem 1.1]{DavisMahowald}), so the dashed map is surjective.

From this it follows that the induced map $\check{\Gamma}_g^{\mathrm{sfr},\ell_W}/\Theta_{4k-1}^\mathrm{sfr} \to \Gamma^{\mathrm{sfr}, [\ell_W]}_g/\Theta_{4k-1}$ is an isomorphism.

By \cite[Lemma 6.3]{KRWframings} we have $\Gamma^{\mathrm{sfr}, [\ell_W]}_g = \Gamma^{\mathrm{sfr}, [[\ell_W]]}_g$, which is the stabiliser of $\ell_W$ for the action of $\Gamma_g$ on the set of stable framings on $W_{g,1}$ relative to a point. Thus $\Gamma^{\mathrm{sfr}, [\ell_W]}_g/\Theta_{4k-1}$ is the stabiliser for the induced action of $\Gamma_g/\Theta_{4k-1}$ on this set. This action is analysed in \cite[Section 2]{K}, and it follows (using \cite[Lemma 2.1]{K}) that the subgroup $I_g/\Theta_{4k-1} < \Gamma_g/\Theta_{4k-1}$ acts freely, and hence that the homomorphism 
$$\Gamma^{\mathrm{sfr}, [[\ell_W]]}_g/\Theta_{4k-1} \lra \Sp_{2g}(\bZ)$$
is injective. But the image $G_g^{\mathrm{sfr}, [[\ell_W]]}$ of $\Gamma^{\mathrm{sfr}, [[\ell_W]]}_g$ in $\Sp_{2g}(\bZ)$ is described in \cite[Section 6]{KRWframings}, and is $\Sp_{2g}^q(\bZ)$ for $2k-1 \neq 3, 7$ and $\Sp_{2g}^q(\bZ)$ or $\Sp_{2g}^a(\bZ)$ if $2n-1 = 3, 7$ depending on the Arf invariant of a certain quadratic form associated to the framed manifold $(W_{g,1}, \ell_W)$, see \cite[Sections 2.5.2, 2.5.3]{KRWframings}.
\end{proof}

Assigning to an element of
$\Gamma^{\mathrm{sfr}}_\partial(W_{g,1},\ell_W)$ the stably-framed
manifold given by its mapping torus defines a homomorphism
$\Gamma^{\mathrm{sfr}}_\partial(W_{g,1},\ell_W) \to \pi_{4k-1}^s$,
extending $\Theta_{4k-1}^\mathrm{sfr} \to \pi_{4k-1}^s$. Together with
the homomorphism described in Lemma \ref{lem:Spq} this yields a central extension
\begin{equation}\label{eq:DeligneBig}\tag{1}
0 \lra \bZ=\mathrm{bP}_{4k}^\mathrm{sfr} \lra \Gamma^{\mathrm{sfr}}_\partial(W_{g,1},\ell_W) \lra \pi_{4k-1}^s \times \Sp_{2g}^{q \text{ or } a}(\bZ)  \lra 0.
\end{equation}
Following \cite[Section 7]{GRWAb}, we may understand something about this extension by knowing the abelianisation of $\Gamma^{\mathrm{sfr}}_\partial(W_{g,1},\ell_W)$, which for large enough $g$ can be determined by applying \cite[Corollary 1.8]{GRW}.  Namely, the corresponding Madsen--Tillmann spectrum is $\Sigma^{\infty-(4k-2)} \OO/\OO(4k-2)_+$ and using \cite[Lemma 5.2]{GRWAb} it follows that the abelianisation is 
$$\pi_{4k-1}^s(\Sigma^\infty \OO/\OO(4k-2)_+) = \pi_{4k-1}^s \oplus \bZ/4.$$
This is the same as the abelianisation of $\pi_{4k-1}^s \times \Sp_{2g}^{q \text{ or } a}(\bZ)$ by e.g.\ \cite[Section 4.1]{KRWframings}, so by considering the Serre spectral sequence for this extension it follows that the class $e \in H^2(\pi_{4k-1}^s \times \Sp_{2g}^{q \text{ or } a}(\bZ) ; \bZ)$ classifying it yields a homomorphism
$$e(-) : H_2(\pi_{4k-1}^s \times \Sp_{2g}^{q \text{ or } a}(\bZ) ; \bZ) \lra \bZ$$
that is surjective. As $\pi_{4k-1}^s$ is a finite group, its presence does not affect the image of this map. It follows that restricting \eqref{eq:DeligneBig} to the second factor defines a central extension
\begin{equation}\label{eq:DeligneExt}\tag{2}
0 \lra \bZ=\mathrm{bP}_{4k}^\mathrm{sfr} \lra E \lra \Sp_{2g}^{q \text{ or } a}(\bZ) \lra 0
\end{equation}
whose Euler class is such that $H_2(\Sp_{2g}^{q \text{ or } a}(\bZ);\bZ) \to \bZ$ is surjective, and this $E$ is a finite index normal subgroup of $\Gamma^{\mathrm{sfr}}_\partial(W_{g,1},\ell_W)$.

As recalled in \cite[Section 4.1]{KRWframings} the groups $\Sp_{2g}^q(\bZ)$ and $\Sp_{2g}^a(\bZ)$ enjoy homological stability with respect to $g$, and have the same stable homology, so calculations stated for the first kind are also valid for the second kind. It then follows from \cite[Lemma 3.15, (3.10), Definition 3.17]{K} (or \cite[Lemma 7.5, proof of Theorem 7.7]{GRWAb}) that for large enough $g$ we have 
$$H^2(\Sp_{2g}^{q \text{ or } a}(\bZ) ; \bZ) = \bZ\{\tfrac{\mathrm{sgn}}{8}\} \oplus \bZ/4,$$
so $e = \pm \tfrac{\mathrm{sgn}}{8}$ modulo torsion. Using \cite[Lemma 5.1 (iii)]{KKM} the finite residual of $E$ is the same as the finite residual of the extension corresponding to $\tfrac{\mathrm{sgn}}{8}$, and by \cite[Lemma 5.2]{KKM} this is the subgroup $\bZ = \mathrm{bP}_{4k}^\mathrm{sfr} < E$.

As $E$ has finite index in $\Gamma^{\mathrm{sfr}}_\partial(W_{g,1},\ell_W)$ it follows that the finite residual of the latter group is also the subgroup $\bZ = \mathrm{bP}_{4k}^\mathrm{sfr}$, so every element of $\mathrm{bP}_{4k-1} \leq \pi_0\mathrm{Diff}^+(W_g)$ lifts to an element of $\Gamma^{\mathrm{sfr}}_\partial(W_{g,1},\ell_W)$  which lies in the finite residual of this group, as required.

\subsection*{The framed mapping class group}

To extend this to the framed, rather than stably-framed, mapping class group, we make use of the following lemma. We are grateful to Henry Wilton for advice on it.

\begin{lemma}\label{lem:AscendingFinRes}
Let
$$1 \lra F \lra G \lra H \lra 1$$
be a group extension with $F$ finite. Then the induced map $\fr(G) \to \fr(H)$ of finite residuals is surjective.
\end{lemma}
\begin{proof}
Let $\iota(G) \leq \widehat{G}$ be the image of $G$ in its profinite completion. The image $F'$ of the finite group $F$ in $\widehat{G}$ is finite, and is therefore closed in the profinite topology (as this is Hausdorff and finite subspaces are compact and so closed). The normaliser of a closed subgroup is closed in the profinite topology: the normaliser of $F'$ contains $\iota(G)$ and so contains its closure $\widehat{G}$, and hence $F'$ is normal in $\widehat{G}$. Then $\widehat{G}/F'$ is the quotient by a closed normal subgroup so is again a profinite group, and we obtain a commutative diagram
\begin{equation*}
\begin{tikzcd}
1 \rar & F \rar \dar & G \rar{\pi} \dar{\iota} & H \rar \dar{\iota} & 1\\
1 \rar & F' \rar & \widehat{G} \rar{p} & \widehat{G}/F' \rar & 1.
\end{tikzcd}
\end{equation*}
Let $h \in \fr(H)$ and $g \in G$ be such that $\pi(g)=h$. Now $\iota(h)=1$ as $\widehat{G}/F'$ is profinite and hence residually finite. It follows that $\iota(g) \in F'$. Thus there is an $f \in F$ such that $\iota(f^{-1} \cdot g)=1$, i.e. $f^{-1} \cdot g \in \fr(G)$. As $\pi(f^{-1} \cdot g) = \pi(g)=h$, this proves the claim.
\end{proof}

After choosing a reference framing of $W_{g,1}$, the map from the space of framings to the space of stable framings corresponds to
$$\map_\partial(W_{g,1}, \mathrm{SO}(4k-2)) \lra \map_\partial(W_{g,1},\mathrm{SO}).$$
\begin{lemma}
The homotopy fibre $F$ of this map over any point has $\pi_0(F)= \bZ/4$ and $\pi_1(F)$ finite.
\end{lemma}
\begin{proof}
This is a map of $H$-spaces so the homotopy fibres over any point are the same: we may therefore take fibres over the basepoint. The analogous map
$$\prod_{i=1}^{2g} \Omega^{2k-1} \mathrm{SO}(4k-2) \simeq \map_*(W_{g,1}, \mathrm{SO}(4k-2)) \lra \map_*(W_{g,1}, \mathrm{SO}) \simeq \prod_{i=1}^{2g} \Omega^{2k-1} \mathrm{SO}$$
with relaxed boundary conditions is $(2k-2)$-connected, so in this range $F$ agrees with the homotopy fibre $\Omega^{4k-1}\mathrm{SO}/\mathrm{SO}(4k-2)$ of
$$\Omega^{4k-2}\mathrm{SO}(4k-2) \lra \Omega^{4k-2}\mathrm{SO}.$$
Now $\pi_{4k-1}(\mathrm{SO}/\mathrm{SO}(4k-2)) = \bZ/4$ by \cite[Lemma 5.2]{GRWAb}, and it follows by elementary rational homotopy theory that $\pi_{4k}(\mathrm{SO}/\mathrm{SO}(4k-2))$ is finite.
\end{proof}

The homotopy fibre of
$$B\mathrm{Diff}_\partial^\mathrm{fr}(W_{g,1} ; \ell_\partial^\mathrm{fr}) \lra B\mathrm{Diff}_\partial^\mathrm{sfr}(W_{g,1} ; \ell_\partial^\mathrm{sfr})$$
over each point therefore has the same homotopy groups, leading to an exact sequence
$$\{\text{finite group}\} \lra \Gamma_\partial^\mathrm{fr}(W_{g,1}, \ell^\mathrm{fr}) \lra \Gamma_\partial^\mathrm{sfr}(W_{g,1}, \ell^\mathrm{sfr}) \overset{\curvearrowright}\lra \bZ/4.$$
This shows that $\Gamma_\partial^\mathrm{fr}(W_{g,1}, \ell^\mathrm{fr})$ surjects with finite kernel onto a finite index subgroup $H \leq \Gamma_\partial^\mathrm{sfr}(W_{g,1}, \ell^\mathrm{sfr})$. This must contain the finite residual $\bZ=\mathrm{bP}_{4k}^\mathrm{sfr}$ of $\Gamma_\partial^\mathrm{sfr}(W_{g,1}, \ell^\mathrm{sfr})$ as its finite residual, and by Lemma \ref{lem:AscendingFinRes} it follows that the finite residual of $\Gamma_\partial^\mathrm{fr}(W_{g,1}, \ell^\mathrm{fr})$ surjects onto $\bZ=\mathrm{bP}_{4k}^\mathrm{sfr} \leq \Gamma_\partial^\mathrm{sfr}(W_{g,1}, \ell^\mathrm{sfr})$, as required.
\bibliographystyle{amsalpha}
\bibliography{biblio}

\end{document}